\documentclass[a4paper]{amsart}

\usepackage{amsmath,amssymb,amscd,amsthm,enumerate}
\usepackage{subfigure,graphicx}

\newtheorem{Def}{Definition}[section]
\newtheorem{Th}[Def]{Theorem}
\newtheorem{Prop}[Def]{Proposition}

\newtheorem{Rem}[Def]{Remark}
\newtheorem{Ex}[Def]{Example}
\newtheorem{Nota}[Def]{Notation}

\newcommand{\R}{\mathbb{R}}
\newcommand{\Z}{\mathbb{Z}}

\newcommand{\C}{\mathbb{C}}
\renewcommand{\P}{\mathbb{P}}

\newcommand{\CC}{\mathcal{C}}
\newcommand{\CE}{\mathcal{E}}
\newcommand{\DS}{\displaystyle }

\newcommand{\al}{\alpha }
\newcommand{\be}{\beta }
\newcommand{\ga}{\gamma }
\newcommand{\de}{\delta }
\newcommand{\Ga}{\Gamma }
\newcommand{\De}{\Delta }

\newcommand{\vep}{\varepsilon }
\newcommand{\vph}{\varphi }
\newcommand{\la}{\lambda }
\newcommand{\om}{\omega }
\newcommand{\si}{\sigma }
\newcommand{\na}{\nabla }
\newcommand{\pa}{\partial }

\newcommand{\bu}{\bullet}

\begin{document}

\title[Twisted cycles of $F_C$]
{Twisted cycles and twisted period relations \\
for Lauricella's hypergeometric function $F_C$
}

\author[Y. Goto]
{Yoshiaki Goto}

\address{Department of Mathematics, Hokkaido University, Sapporo 060-0810, Japan }
\email{y-goto@math.sci.hokudai.ac.jp}

\dedicatory{}

% \author{OTHER D. AUTHOR}

% \address{Full affiliations \\
% and mailing addresses}

\maketitle

\begin{abstract}
We study Lauricella's hypergeometric function $F_C$ by 
using twisted (co)homology groups. 
We construct twisted cycles with respect to an 
integral representation of Euler type of $F_C$. 
These cycles correspond to $2^m$ linearly independent solutions to 
the system of differential equations annihilating $F_C$. 
Using intersection forms of twisted (co)homology groups, we obtain 
twisted period relations which give quadratic relations for Lauricella's $F_C$.
\end{abstract}

\section{Introduction}
Lauricella's hypergeometric series $F_C$ of $m$-variables $x_1 ,\ldots ,x_m$ 
with complex parameters $a,b,c_1 ,\dots ,c_m$ is defined by 
\begin{align*}
 F_{C} (a,b,c ;x ) 
 =\sum_{n_{1} ,\ldots ,n_{m} =0} ^{\infty } 
 \frac{(a,n_{1} +\cdots +n_{m} )(b,n_{1} +\cdots +n_{m} )}
 {(c_{1} ,n_{1} )\cdots (c_{m} ,n_{m} ) n_{1} ! \cdots n_{m} !} x_{1} ^{n_{1}} \cdots x_{m} ^{n_{m}} ,
\end{align*}
where $x=(x_1 ,\ldots ,x_m),\ c=(c_1 ,\ldots ,c_m)$, 
$c_1 ,\ldots ,c_m \not\in \{ 0,-1,-2,\ldots \}$ and $(c_1 ,n_1)=\Gamma (c_1+n_1)/\Gamma (c_1)$. 
This series converges in the domain 
$$
D_{C} :=\left\{ (x_{1} ,\ldots ,x_{m} ) \in \C ^m  \ \middle| \ \sum _{k=1} ^{m} \sqrt{|x_{k}|} <1  \right\} ,
$$
and admits the integral representation (\ref{integral}). 
The system $E_C (a,b,c)$ of differential equations annihilating $F_C (a,b,c;x)$ is a holonomic system 
of rank $2^m$ with the singular locus $S$ given in (\ref{sing-locus}). 
There is a fundamental system of solutions to $E_C (a,b,c)$ in a simply connected domain in $D_C -S$, 
which is given in terms of Lauricella's hypergeometric series $F_C$ with different parameters, 
see (\ref{series-sol}) for their expressions. 

In the case of $m=2$, the series $F_C(a,b,c;x)$ and the system $E_C(a,b,c)$ are called 
Appell's hypergeometric series $F_4(a,b,c;x)$ and system $E_4(a,b,c)$ of differential equations, 
which are studied in \cite{GM} by twisted (co)homology groups concerning with 
the integral representation (\ref{integral}) for $m=2$.
We construct twisted cycles corresponding to four solutions to $E_4(a,b,c)$ expressed by 
the series $F_4$, and evaluate their intersection numbers. We also evaluate 
the intersection matrix for a basis of the twisted cohomology group. 
By using these results, we determine the monodromy representation of $E_4(a,b,c)$ 
and give twisted period relations, which are quadratic relations between two 
fundamental systems of $E_4$ with different parameters. 

In this paper, we construct $2^m$ twisted cycles which represent elements 
of the $m$-th twisted homology group 
concerning with the integral representation (\ref{integral}). 
They imply integral representations of the solutions (\ref{series-sol}) 
expressed by the series $F_C$. 
We evaluate the intersection numbers of these $2^m$ twisted cycles. 
We also evaluate the intersection numbers of some elements of 
the twisted cohomology group. 
These results imply twisted period relations for two fundamental systems of $E_C$ 
with different parameters, refer to Theorem \ref{TPR} for their explicit forms. 

In the study of twisted homology groups, twisted cycles 
given by bounded chambers are useful. 
However there are few bounded chambers in our case. 
By avoiding this difficulty, we succeed in constructing $2^m$ twisted cycles. 
We explain our idea in the construction. 
Our twisted homology group is defined by the multi-valued function 
\begin{align*}
u(t) =& \prod_{k=1}^m t_k ^{1-c_k +b} \cdot v^{\sum_{k=1}^m c_{k} -a-m+1} \cdot w ^{-b} ,\\
  & v=1-\sum_{k=1}^m t_k ,\ w= \prod_{k=1}^m t_k \cdot (1-\sum_{k=1}^m \frac{x_k}{t_k}) ,
\end{align*}
for fixed small positive real numbers $x_1 ,\ldots ,x_m$. 
Let $\{ i_1 ,\ldots ,i_r \}$ be a subset of $\{ 1,\ldots ,m \}$ of cardinality $r$ 
and let $\{ j_1 ,\ldots ,j_{m-r} \}$ be its complement. 
We embed the direct product of an $r$-simplex $\tau_{i_1 \cdots i_r}$ and
an $(m-r)$-simplex $\tau' _{j_1 \cdots j_{m-r}}$ into the bounded chamber 
$$
\{ (t_1 ,\ldots ,t_m) \in \R^m \mid t_1 ,\ldots ,t_m ,\ v,\ w >0 \} .
$$
We consider an $m$-dimensional twisted chain given by the image of 
$\tau_{i_1 \cdots i_r} \times \tau' _{j_1 \cdots j_{m-r}}$ with 
loading a branch of $u(t)$ on it. 
As is in Section 3.2.4 of \cite{AK}, 
we can eliminate its boundary by regarding the image 
of $\tau_{i_1 \cdots i_r}$ as the simplex 
$$
\{ (s_{i_1} ,\ldots ,s_{i_r}) \in \R^r \mid 
s_{i_1},\ldots ,s_{i_r} ,1-\sum_{p=1}^{r} s_{i_p} \geq \vep \} ,
$$
and that of $\tau'_{j_1 \cdots j_{m-r}}$ as the simplex 
$$
\{ (t_{j_1} ,\ldots ,t_{j_{m-r}}) \in \R^{m-r} \mid 
t_{j_1},\ldots ,t_{j_{m-r}} ,1-\sum_{q=1}^{m-r} t_{j_q} \geq \vep \} ,
$$
where $s_{i_p}=x_{i_p}/t_{i_p}$ for $p=1,\ldots ,r$ and 
$\vep$ is a certain positive real number. 
In this elimination of the boundary, we must regulate the difference of branches 
of $u(t)$ by a different way from the usual regularization. 
We realize this elimination by using the twisted homology group defined by 
another multi-valued function, see Section \ref{section-cycle} for details. 
Our first main theorem states that this twisted cycle corresponds to 
the solution to $E_C(a,b,c)$ with the power function $\prod_{p=1}^r x_{i_p} ^{1-c_{i_p}}$. 
Our construction of twisted cycles is also useful in the study on the system of 
differential equations annihilating Lauricella's hypergeometric series $F_A$. 
Refer to \cite{G-FA}, 
for the twisted cycles corresponding to the solutions expressed by the series $F_A$ 
and twisted period relations for $F_A$. 

By our first main theorem and the proof of Lemma 4.1 in \cite{GM}, 
it turns out that the intersection matrix becomes diagonal. 
Moreover, our construction and results in \cite{KY} enable us to evaluate 
the diagonal entries of the intersection matrix. 
If the intersection matrix for bases of twisted homology groups is 
evaluated, then the intersection numbers of some elements of 
twisted cohomology groups imply twisted period relations, 
which are originally identities among the integrals given by 
the pairings of elements of twisted homology and cohomology groups. 
Our first main theorem transforms these identities into 
quadratic relations among hypergeometric series $F_C$'s. 
Our second main theorem states these formulas in Section \ref{section-TPR}. 

As is in \cite{HT}, the irreducibility condition of the system $E_C(a,b,c)$ is 
known to be 
\begin{align*}
  a-\sum _{p=1}^{r} c_{i_r} ,\ \ b-\sum _{p=1}^{r} c_{i_r} \not\in \Z 
\end{align*}
for any subset $\{ i_1 ,\ldots ,i_r \}$ of $\{ 1,\ldots ,m \}$. 
Since our interest is in the property of solutions to $E_C (a,b,c)$ 
expressed in terms of the hypergeometric series $F_C$, 
we assume throughout this paper that the parameters $a,\ b$  and 
$c=(c_1 ,\ldots ,c_m)$ 
satisfy the condition above and $c_1 ,\ldots ,c_m \not\in \Z$.

\section{Differential equations and integral representations}
In this section, we collect some facts about Lauricella's $F_C$ 
and the system $E_C$ of 
hypergeometric differential equations annihilating it. 

\begin{Nota}\label{k}
Throughout this paper, the letter $k$ always stands for an index running from $1$ to $m$. 
If no confusion is possible, 
$\DS \sum_{k=1} ^m$ and $\DS \prod_{k=1} ^m$ are often simply denoted by 
$\sum$ (or $\sum_k$) and $\prod$ (or $\prod_k$), respectively. 
For example, under this convention  
$F_C (a,b,c;x)$ is expressed as 
\begin{align*}
F_{C} (a,b,c ;x ) 
=\sum_{n_{1} ,\ldots ,n_{m} =0} ^{\infty } 
 \frac{(a,\sum n_{k} )(b,\sum n_{k} )}
 {\prod (c_{k} ,n_{k} )\cdot \prod n_{k} ! } \prod x_{k} ^{n_{k}} . 
\end{align*}
\end{Nota}

Let $\pa_k \ (k=1,\dots ,m)$ be the partial differential operator with respect to $x_k$. 
Lauricella's $ F_C (a,b,c;x)$ satisfies hypergeometric differential equations 
\begin{align*}
& \Bigl[
x_{k} (1-x_{k} )\pa _{k} ^{2} -x_{k} \sum _{\stackrel{1\leq i\leq m}{i\neq k}} x_{i} \pa _{i} \pa _{k}
-\sum _{\stackrel{1\leq i,j\leq m}{i\neq k}} x_{i} x_{j} \pa _{i} \pa _{j} \\
& +(c_{k} -(a+b+1)x_{k} )\pa _{k} -(a+b+1)\sum _{\stackrel{1\leq i\leq m}{i\neq k}} x_{i} \pa _{i} -ab
\Bigr] f(x)=0, 
\end{align*}
for $k=1,\ldots ,m$. 
The system generated by them is 
called Lauricella's system $E_C (a,b,c)$ of hypergeometric differential equations. 

\begin{Prop}[\cite{HT}, \cite{L}] \label{solution}
The system $E_{C} (a,b,c)$ is a holonomic system of rank $2^m$ with the singular locus 
\begin{align}
\label{sing-locus}
S:= \left( 
\prod_k x_{k} \cdot \prod _{\vep _{1} ,\ldots ,\vep _{m} =\pm 1} (1+\sum_{k} \vep _{k} \sqrt{x_{k}})=0
\right) \subset \C^m . 
\end{align}
If $c_1 ,\ldots ,c_m \not\in \Z$, then 
the vector space of solutions to $E_{C} (a,b,c)$ in a simply connected domain in 
$D_C -S$ is spanned by the following $2^m$ elements: 
\begin{align}
\label{series-sol} 
f_{i_{1} \cdots i_{r}} 
:=\prod _{p=1} ^{r} x_{i_{p}} ^{1-c_{i_{p}}} \cdot 
F_{C} \left( a+r-\sum _{p=1} ^{r} c_{i_{p}} ,b+r-\sum _{p=1} ^{r} c_{i_{p}} ,c^{i_1 \cdots i_r} ;
x \right) . 
\end{align}
Here $r$ runs from $0$ to $m$, indices $i_{1} ,\ldots ,i_{r}$ satisfy $1\leq i_{1} <\cdots <i_{r} \leq m$, 
and the row vector $c^{i_1 \cdots i_r}$ is defined by 
$$
c^{i_1 \cdots i_r} :=c+2\sum_{p=1}^r (1-c_{i_p})e_{i_p} ,
$$
where $e_{i}$ is the $i$-th unit row vector of $\C^m$.
\end{Prop}
For the above $i_1 ,\ldots ,i_r$, we take $j_{1} ,\ldots ,j_{m-r}$ so that 
$1\leq j_{1} <\cdots <j_{m-r} \leq m$ and $\{ i_{1} ,\ldots ,i_{r} ,j_{1} ,\ldots ,j_{m-r} \} =\{1,\ldots ,m \}$. 
It is easy to see that the $i_p$-th entry of $c^{i_1 \cdots i_r}$ is $2-c_{i_p}$ ($1\leq p \leq r$) and 
the $j_q$-th entry is $c_{j_q}$ ($1 \leq q \leq m-r$). 

We denote the multi-index `` $i_1 \cdots i_r$'' by a letter $I$ 
expressing the set $\{ i_1,\ldots ,i_r \}$. 
Note that the solution (\ref{series-sol}) for $r=0$ is 
$f(=f_{\emptyset})=F_{C} (a,b,c ;x)$. 

\begin{Prop}[Integral representation of Euler type, 
Example 3.1 in \cite{AK}] \label{int-Euler}
For sufficiently small positive real numbers $x_1 ,\ldots ,x_m$, 
if $c_1 ,\ldots ,c_m ,a-\sum c_k \not\in \Z$, 
then $F_C (a,b,c;x)$ admits the following integral representation: 
\begin{align}
& F_{C} (a,b,c_{1} ,\ldots ,c_{m} ;x_{1} ,\ldots ,x_{n} ) \label{integral} \\
& =\frac{\Ga (1-a)}{\prod \Ga (1-c_{k})\cdot \Ga (\sum c_{k} -a-m+1)} \nonumber \\
& \ \ \cdot \int _{\De } \prod t_{k} ^{-c_{k} } \cdot (1-\sum t_{k} )^{\sum c_{k} -a-m} 
	\cdot \left( 1-\sum \frac{x_{k}}{t_{k}} \right) ^{-b} dt_{1} \wedge \cdots \wedge dt_{m} , 
         \nonumber
\end{align}
where $\De $ is the twisted cycle made by an $m$-simplex in 
Sections 3.2 and 3.3 of \cite{AK}. 
\end{Prop}
In fact, this cycle is one of twisted cycles constructed in Section \ref{section-cycle}. 

\section{Twisted homology groups}\label{section-THG}
We review twisted homology groups and the intersection form 
between twisted homology groups in general situations, 
by referring to Chapter 2 of \cite{AK} and 
Chapters IV, VIII of \cite{KY}. 

For polynomials $P_j(t)=P_j(t_1 ,\ldots ,t_m ) \ (1 \leq j \leq n)$, 
we set $D_j :=\{ t \mid  P_j(t)=0 \} \subset \C^m$ and $M:=\C^m -(D_1 \cup \cdots \cup D_n)$.
We consider a multi-valued function $u(t)$ on $M$ defined as 
$$
u(t):=\prod _{j=1}^{n} P_j(t)^{\la _j} ,\ \la_j \in \C -\Z \ (1 \leq j \leq n).
$$
Let $\si$ be a $k$-simplex in $M$, we define a loaded $k$-simplex $\si \otimes u$ by
$\si$ loading a branch of $u$ on it. 
We denote the $\C$-vector space of finite sums of loaded $k$-simplexes by $\CC_k (M,u)$, 
called the $k$-th twisted chain group. 
An element of $\CC_k(M,u)$ is called a twisted $k$-chain. 
For a loaded $k$-simplex $\si \otimes u$ and a smooth $k$-form $\vph$ on $M$, 
the integral $\int_{\si \otimes u} u \cdot \vph$ is defined by 
$$
\int_{\si \otimes u} u \cdot \vph :=
\int_{\si} \left[ {\rm the \ fixed \ branch} \ {\rm of} \ u \ {\rm on} \ \si \right] \cdot \vph .
$$
By the linear extension of this, we define the integral on a twisted $k$-chain.  

We define the boundary operator $\pa^u :\CC_k (M,u) \to \CC_{k-1} (M,u)$ by
$$
\pa^u (\si \otimes u):= \pa(\si) \otimes u|_{\pa(\si)} , 
$$
where $\pa$ is the usual boundary operator and 
$u|_{\pa(\si)}$ is the restriction of $u$ to $\pa (\si)$. 
It is easy to see that $\pa^u \circ \pa^u =0$. 
Thus we have a complex 
$$
\CC_{\bu} (M,u) :\cdots 
\overset{\pa^u}{\longrightarrow} \CC_k (M,u)
\overset{\pa^u}{\longrightarrow} \CC_{k-1}(M,u)
\overset{\pa^u}{\longrightarrow} \cdots ,
$$
and its $k$-th homology group $H_k(\CC_{\bu} (M,u))$. 
It is called the $k$-th twisted homology group. 
An element of $\ker \pa^u$ is called a twisted cycle.

By considering $u^{-1} =1/u$ instead of $u$, we have $H_k(\CC_{\bu} (M,u^{-1}))$. 
There is the intersection pairing $I_h$ between $H_m(\CC_{\bu} (M,u))$ and $H_m(\CC_{\bu} (M,u^{-1}))$
(in fact, the intersection pairing is defined between $H_k(\CC_{\bu} (M,u))$ 
and $H_{2m-k}(\CC_{\bu} (M,u^{-1}))$, however we do not consider the cases $k \neq m$). 
Let $\De$ and $\De'$ be elements of $H_m(\CC_{\bu} (M,u))$ and $H_m(\CC_{\bu} (M,u^{-1}))$ given by 
twisted cycles $\sum_i \al_i \cdot \si_i \otimes u_i$ and $\sum_j \al'_j \cdot \si'_j \otimes u_j ^{-1}$ 
respectively,  
where $u_i$ (resp. $u_j ^{-1}$) is a branch of $u$ (resp. $u^{-1}$) on $\si_i$ (resp. $\si'_j$). 
Then their intersection number is defined by 
$$
I_h (\De ,\De'):=\sum_{i,j} \sum_{s \in \si_i \cap \si'_j} \al_i \al'_j 
\cdot (\si_i \cdot \si'_j)_s \cdot \frac{u_i (s)}{u_j(s)} ,
$$
where $(\si_i \cdot \si'_j)_s$ is the topological intersection number of 
$m$-simplexes $\si_i$ and $\si'_j$ at $s$. 

~

In this paper, we mainly consider 
$$
M:=\C ^{m} -\left( H_{1} \cup \cdots \cup H_{m} \cup H \cup D \right) ,
$$
where 
\begin{align*}
  & H_k :=(t_k =0) \ (1 \leq k \leq m),\ H:=\left( v =0\right) ,\ 
  D:=\left( w =0 \right) ,\\
  & v:=1-\sum t_k ,\ w:= \prod t_k \cdot (1-\sum \frac{x_k}{t_k}) .
\end{align*}
Note that $w$ is a polynomial in $t_1 ,\ldots ,t_m$. 
We consider the twisted homology group on $M$ with respect to the multi-valued function 
\begin{align*}
u :=& \prod t_k ^{1-c_k +b} \cdot v^{\sum c_{k} -a-m+1} w ^{-b} \\
=& \prod t_k ^{1-c_k } \cdot (1-\sum t_{k} )^{\sum c_{k} -a-m+1} 
\cdot \left( 1-\sum \frac{x_{k}}{t_{k}} \right) ^{-b} 
\end{align*}
(the second equality holds under the coordination of branches). 
Proposition \ref{int-Euler} means that the integral 
$$
\int_{\De} u \vph ,\ \ 
\vph  :=\frac{dt_{1} \wedge \cdots \wedge dt_{m} }{\prod t_k \cdot (1-\sum t_{k} )} 
$$
represents $F_C(a,b,c;x)$ modulo Gamma factors.

\section{Twisted cycles corresponding to local solutions 
$f_{i_1 \cdots i_r}$}\label{section-cycle}
In this section, we construct $2^m$ twisted cycles in $M$ corresponding to 
the solutions (\ref{series-sol}) to $E_C (a,b,c)$.

Let $0 \leq r \leq m$ and subsets $\{ i_1 ,\ldots ,i_r \}$ and $\{ j_1 ,\ldots ,j_{m-r} \}$ 
of $\{ 1,\ldots ,m \}$ satisfy 
$i_1 <\cdots <i_r ,\ j_{1} <\cdots <j_{m-r}$ and 
$\{ i_{1} ,\ldots ,i_{r} ,j_{1} ,\ldots ,j_{m-r} \} =\{1,\ldots ,m \}$. 
\begin{Nota}
From now on, the letter $p$ (resp. $q$) is always stands for an index running 
from $1$ to $r$ (resp. from $1$ to $m-r$). 
We use the abbreviations $\sum ,\ \prod$ for the indices 
$p,q$ as are mentioned in Notation \ref{k}. 
\end{Nota}

We set 
\begin{align*}
  M _{i_1 \cdots i_r} =\C ^m 
  -\left( \bigcup _{k} (s_k =0) \cup (v_{i_1 \cdots i_r} =0) \cup (w_{i_1 \cdots i_r} =0) \right) ,
\end{align*}
where 
\begin{align*}
  v_{i_1 \cdots i_r} :=\prod_p s_{i_p} \! \cdot (1-\! \sum_p \frac{x_{i_p}}{s_{i_p}}-\! \sum_q s_{j_q}) , \ 
  w_{i_1 \cdots i_r} :=\prod_q s_{j_q} \! \cdot (1-\! \sum_p s_{i_p}-\! \sum_q \frac{x_{j_q}}{s_{j_q}})   
\end{align*}
are polynomials in $s_1 ,\ldots ,s_m$. 
Let $u_{i_1 \cdots i_r}$ and $\vph_{i_1 \cdots i_r}$ be 
a multi-valued function and an $m$-form on $M_{i_1 \cdots i_r}$ defined as 
\begin{align*}
u_{i_{1} \cdots i_{r}} :=\prod _{k} s_{k} ^{C_{k}} 
\cdot v_{i_1 \cdots i_r} ^{A} \cdot w_{i_1 \cdots i_r} ^{B} , \quad 
\vph _{i_1 \cdots i_r} :=\frac{ds_1 \wedge \cdots \wedge ds_m}{s_1 \cdots s_m v_{i_1 \cdots i_r}} , 
\end{align*}
where
\begin{align*}
  A:=\sum c_{k} -a-m+1,\ B:=-b,  \ 
  C_{i_p} :=c_{i_p} -1 -A ,\ C_{j_q} :=1-c_{j_q} -B .
\end{align*}
We construct a twisted cycle $\tilde{\De }_{i_1 \cdots i_r}$ in $M_{i_1 \cdots i_r}$ 
with respect to $u_{i_1 \cdots i_r}$. 
Note that if $\{ i_1 ,\ldots ,i_r \} =\emptyset$, then these settings coincide  
with those in the end of Section \ref{section-THG}. 
We choose positive real numbers $\vep _{1} ,\ldots ,\vep _{m}$ and $\vep$ so that 
$\vep <1-\sum_{k} \vep _k$. 
And let $x_1 ,\ldots ,x_m$ be small positive real numbers satisfying 
$$
x_{k} <\frac{\vep _{k}}{m} \vep 
$$
(for example, if 
$$
\vep _{k} =\vep =\frac{1}{3m} ,\ 0<x_{k} <\frac{1}{9m^3} ,
$$
these conditions hold). 
Thus the closed subset 
\begin{align*}
  \sigma_{i_1 \cdots i_r} :=\left\{ (s_1 ,\ldots ,s_m )\in \R^m \ \Bigg| \ 
    \begin{array}{l}
      s_{i_p} \geq \vep_{i_p} ,\ 1-\sum s_{i_p} \geq \vep ,\\
      s_{j_q} \geq \vep_{j_q} ,\ 1-\sum s_{j_q} \geq \vep
    \end{array}
  \right\} 
\end{align*}
is nonempty, since we have
$(\vep_1 \! +\! \frac{\de}{2m},\ldots ,\vep_m \! +\! \frac{\de}{2m}) \in \sigma_{i_1 \cdots i_r}$, 
where $\de :=1  - \sum \vep_k - \vep >0$. 
Further, $\sigma_{i_1 \cdots i_r}$ is contained in the bounded domain 
$$
\left\{ (s_1 ,\ldots ,s_m )\in \R^m \ \Bigg| \ s_k >0 ,\ 
    \begin{array}{l}
      1-\sum \frac{x_{i_p}}{s_{i_p}} -\sum s_{j_q} >0 ,\\
      1-\sum s_{i_p} -\sum \frac{x_{j_q}}{s_{j_q}} >0
    \end{array}
  \right\} 
\subset (0,1)^m ,
$$
and is a direct product of an $r$-simplex and an $(m-r)$-simplex. 
Indeed, $(s_{1} ,\ldots ,s_{m}) \in \sigma _{i_{1} \cdots i_{r}} $ satisfies 
\begin{align*}
1-\sum \frac{x_{i_p}}{s_{i_p}} -\sum s_{j_q}  &> 
1-\frac{r}{m} \vep -\sum s_{j_q} > 1-\sum s_{j_q} -\vep \geq 0, \\
1-\sum s_{i_p} -\sum \frac{x_{j_q}}{s_{j_q}} &> 
1-\sum s_{i_p} -\frac{m-r}{m} \vep > 1-\sum s_{i_p} -\vep \geq 0 .
\end{align*}
The orientation of $\sigma _{i_{1} \cdots i_{r}} $ 
is induced from the natural embedding $\R^m \subset \C^m$. 
We construct a twisted cycle from $\sigma _{i_{1} \cdots i_{r}}$. 
We may assume that $\vep _{k} =\vep$ (the above example satisfies this condition), and denote them by $\vep$. 
Set $L_{1} :=(s_1 =0),\ldots ,\ L_{m} :=(s_m =0) ,\ L_{m+1} :=(1-\sum s_{i_p}=0),\ L_{m+2} :=(1-\sum s_{j_q} =0)$, 
and let $U(\subset \R ^m )$ be the bounded chamber surrounded by $L_{1} ,\ldots ,\ L_{m} ,\ L_{m+1} ,\ L_{m+2}$, 
then $\sigma _{i_{1} \cdots i_{r}}$ is contained in $U$. 
Note that we do not consider the hyperplane $L_{m+1}$ (resp. $L_{m+2}$), when $r=0$ (resp. $r=m$). 
For $J\subset \{ 1, \ldots ,m+2 \}$, we consider $L_{J} :=\cap _{j\in J} L_{j} ,\ U_{J} :=\overline{U} \cap L_{J}$ 
and $T_J :=\vep $-neighborhood of $U_J$. 
Then we have 
$$
\sigma _{i_{1} \cdots i_{r}} =U-\bigcup _{J} T_{J} .
$$ 
Using these neighborhoods $T_J$, we can construct a twisted cycle $\tilde{\De} _{i_{1} \cdots i_{r}}$ 
in the same manner as Section 3.2.4 of \cite{AK} 
(notations $L$ and $U$ correspond to $H$ and $\De$ in \cite{AK}, respectively). 
Note that we have to consider contribution of branches of
$\DS s_{i_p} ^{C_{i_p}} \cdot v_{i_1 \cdots i_r}^{A}$, 
when we deal with the circle associated to $L_{i_p} \ (p=1,\dots ,r)$. 
Indeed, for fixed positive real numbers $s_k \ (k\neq i_p )$, 
$s_{i_p}$ satisfying $1-\sum \frac{x_{i_{p}}}{s_{i_{p}}} -\sum s_{j_{q}} =0$ belongs to $\R$ and we have 
\begin{align*}
s_{i_{p}} =\frac{x_{i_p}}{1-\sum _{q} s_{j_q} -\sum _{p'\neq p} \frac{x_{i_{p'}}}{s_{i_{p'}}}}
<\frac{\frac{\vep}{m} \vep }{\vep-(r-1) \frac{\vep}{m}} =\frac{\vep}{m-r+1}<\vep .
\end{align*}
Thus the exponent about this contribution is 
$$
C_{i_p} +A=c_{i_p} -1.
$$
The exponents about the contributions 
of the circles associated to $L_{j_q} ,\ L_{m+1} ,\ L_{m+2}$ are also evaluated as 
$$
C_{j_q} +B=1-c_{j_q} ,\ B=-b,\ A=\sum c_k -a-m+1, 
$$ 
respectively. 
We briefly explain the expression of $\tilde{\De}_{i_1 \cdots i_r}$. 
For $j=1,\ldots ,m+2$, let $l_j$ be the $(m-1)$-face of $\si_{i_1 \cdots i_r}$ 
given by $\si_{i_1 \cdots i_r} \cap \overline{T_j}$, 
and let $S_j$ be a positively oriented circle with radius $\vep$ 
in the orthogonal complement of $L_j$ starting from the projection of 
$l_j$ to this space and surrounding $L_j$. 
Then $\tilde{\De} _{i_{1} \cdots i_{r}}$ is written as
\begin{align*}
  \sigma _{i_{1} \cdots i_{r}} 
  +\sum _{ \emptyset \neq J \subset \{ 1,\ldots , m+2\}} 
  \ \prod_{j\in J} \frac{1}{d_j} 
  \cdot \left( \Biggl( \bigcap_{j\in J} l_j \Biggr) \times \prod_{j\in J} S_j \right) ,
\end{align*}
where 
\begin{align*}
  d_{i_p} :=\ga_{i_p} -1,\ d_{j_q} :=\ga_{j_q} ^{-1} -1 ,\ 
  d_{m+1} :=\be^{-1} -1,\ d_{m+2} :=\al^{-1} \prod \ga_k -1, 
\end{align*}
and $\al :=e^{2\pi \sqrt{-1}a} ,\ \be :=e^{2\pi \sqrt{-1}b},\ \ga_k :=e^{2\pi \sqrt{-1}c_k}$. 
Note that we define an appropriate orientation for each 
$(\cap_{j\in J} l_j)\times \prod_{j\in J} S_j$, 
see Section 3.2.4 of \cite{AK} for details.

\begin{Ex}
  We give explicit forms of $\tilde{\De}$ and $\tilde{\De}_1$, for $m=2$. 
  \begin{enumerate}[(i)]
  \item In the case of $I=\emptyset$ ($r=0$), we have 
    \begin{align*}
      \tilde{\De} =&\si +\frac{S_1 \times l_1 }{1-\ga_1 ^{-1}} 
      +\frac{S_2 \times l_2}{1-\ga_2 ^{-1}}  
      +\frac{S_4 \times l_4}{1-\al^{-1} \ga_1 \ga_2}  \\
      &+\frac{S_1 \times S_2}{(1\! - \!\ga_1 ^{-1})(1\! -\! \ga_2 ^{-1})} 
      +\frac{S_2 \times S_4}{(1\! -\! \ga_2 ^{-1})(1\! -\! \al^{-1} \ga_1 \ga_2)}   
      +\frac{S_4 \times S_1}{(1\! -\! \al^{-1} \ga_1 \ga_2 )(1\! -\! \ga_1 ^{-1})} , 
    \end{align*}
    where the $1$-chains $l_j$ satisfy $\pa \si =l_1 +l_2 +l_4$ 
    (see Figure \ref{fig1}), 
    and the orientation of each direct product is induced 
    from those of its components. 
    Note that the face $l_3$ does not appear in this case. 
    \begin{figure}[h]
      \centering{
      \includegraphics[scale=1.0]{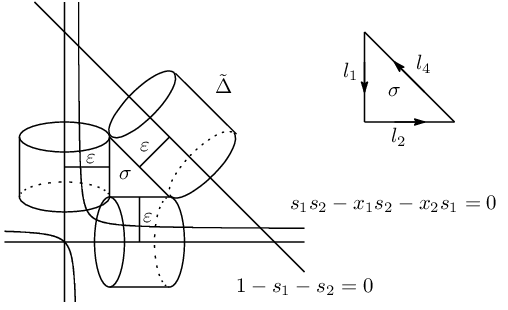} }
      \caption{$\tilde{\De}$ for $m=2$. \label{fig1}}
    \end{figure}
  \item In the case of $I=\{ 1 \}$, we have 
    \begin{align*}
      \tilde{\De}_1 =&\si_1 +\frac{S_1 \times l_1 }{1-\ga_1} 
      +\frac{S_2 \times l_2}{1-\ga_2^{-1}}  
      +\frac{S_3 \times l_3}{1-\be^{-1}} 
      +\frac{S_4 \times l_4 }{1-\al^{-1}\ga_1 \ga_2} \\
      &+\frac{S_1 \times S_2}{(1-\ga_1)(1-\ga_2^{-1})} 
      +\frac{S_2 \times S_3}{(1-\ga_2^{-1})(1-\be^{-1})} \\   
      &+\frac{S_3 \times S_4}{(1-\be^{-1})(1-\al^{-1}\ga_1 \ga_2)}
      +\frac{S_4 \times S_1}{(1-\al^{-1}\ga_1 \ga_2)(1-\ga_1)} , 
    \end{align*}
    where the $1$-chains $l_j$ satisfy $\pa \si =\sum_{j=1}^4 l_j$ 
    (see Figure \ref{fig2}), 
    and the orientation of each direct product is induced 
    from those of its components. 
    \begin{figure}[h]
      \centering{
      \includegraphics[scale=1.0]{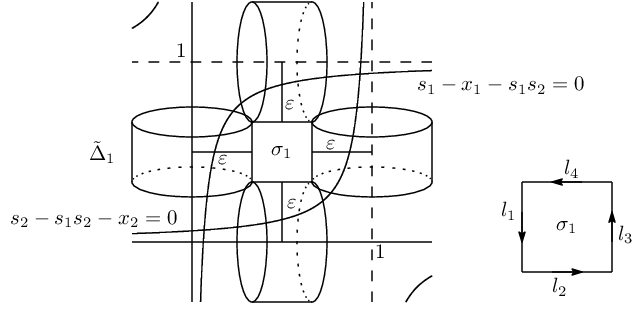} }
      \caption{$\tilde{\De}_1$ for $m=2$. \label{fig2}}
    \end{figure}
  \end{enumerate}
\end{Ex}

We consider the following integrals: 
\begin{align*}
& F _{i_{1} \cdots i_{r}} := 
\int _{\tilde{\De }_{i_{1} \cdots i_{r}}} u_{i_1 \cdots i_r} \vph_{i_1 \cdots i_r} \\ 
&= \int _{\tilde{\De }_{i_{1} \cdots i_{r}}}
\prod _{p=1} ^{r} s_{i_{p}} ^{c_{i_{p}} -2} \cdot \prod _{q=1} ^{m-r} s_{j_{q}} ^{-c_{j_{q}}} 
\cdot \left( 1-\sum _{p=1} ^{r} \frac{x_{i_{p}}}{s_{i_{p}}} -\sum _{q=1} ^{m-r} s_{j_{q}}  \right) ^{\sum c_{k} -a-m} 
\\ & \hspace{20mm} \cdot 
\left(  1-\sum _{p=1} ^{r}  s_{i_{p}} -\sum _{q=1} ^{m-r}\frac{x_{j_{q}}}{s_{j_{q}}}  \right) ^{-b} 
ds_1 \wedge \cdots \wedge ds_m .
\end{align*}

\begin{Prop}\label{series-cycle}
\begin{align*}
F _{i_{1} \cdots i_{r}} 
=& \prod _{p}  \Ga (c_{i_{p}} -1) \! \cdot \! \prod _{q}  \Ga (1-c_{j_{q}} ) 
 \! \cdot \! \frac{\Ga (\sum c_{k} -a-m+1) \Ga (1-b)}{\Ga (\sum c_{i_{p}} -a-r+1) \Ga (\sum c_{i_{p}} -b-r+1)} \\
& \cdot F_{C} \left( a+r-\sum _{p=1} ^{r} c_{i_{p}} ,b+r-\sum _{p=1} ^{r} c_{i_{p}} ,c^{i_1 \cdots i_r} ;x \right) .
\end{align*}
\end{Prop}
\begin{proof}
We compare the power series expansions of the both sides. 
Note that the coefficient of $x_1 ^{n_1} \cdots x_m ^{n_m}$ in the series expression of 
$F_{C} (a+r-\sum _{p} c_{i_{p}} ,b+r-\sum _{p} c_{i_{p}} ,c^{i_1 \cdots i_r} ;x)$ is 
\begin{align*}
A_{n_1 \dots n_m} :=& \frac{\Ga (a+r-\sum_p c_{i_{p}} +\sum_k n_k )}{\Ga (a+r-\sum_p c_{i_{p}} )}
\cdot \frac{\Ga (b+r-\sum_p c_{i_{p}} +\sum_k n_k )}{\Ga (b+r-\sum_p c_{i_{p}} )} \\
& \cdot \prod_p \frac{\Ga (2-c_{i_p})}{\Ga (2-c_{i_p} +n_{i_p})} \cdot \prod_q \frac{\Ga (c_{j_q})}{\Ga (c_{j_q} +n_{j_q})} 
\cdot \prod_k \frac{1}{n_k !}. 
\end{align*}
On the other hand, we have
\begin{align*}
& \left( 1-\sum _{p}  \frac{x_{i_{p}}}{s_{i_{p}}} -\sum _{q}  s_{j_{q}}  \right) ^{\sum c_{k} -a-m} \\
&= \! \! \! \sum _{n_{i_1} ,\dots ,n_{i_r}} \frac{\Ga (a\! -\! \sum_k c_{k} \! +\! m\! +\! \sum_p n_{i_p} )}
       {\Ga (a\! -\! \sum_k c_{k} \! +\! m) \cdot \prod_p n_{i_p} !} 
	\prod_p s_{i_p} ^{-n_{i_p}} \! \cdot \! 
        (1\! -\! \sum_q s_{j_q})^{\sum c_{k} \! -\! a\! -\! m\! -\! \sum n_{i_p}} 
        \! \cdot \! \prod_p x_{i_p} ^{n_{i_p}}
\end{align*}
and
\begin{align*}
& \left(  1-\sum _{p}   s_{i_{p}} -\sum _{q} \frac{x_{j_{q}}}{s_{j_{q}}}  \right) ^{-b} \\
&= \sum _{n_{j_1} ,\dots ,n_{j_{m-r}}} \frac{\Ga (b+\sum_q n_{j_q} )}{\Ga (b) \cdot \prod_q n_{j_q} !} 
	\prod_q s_{j_q} ^{-n_{j_q}} \cdot (1-\sum_p s_{i_p})^{-b-\sum n_{j_q}} \cdot \prod_q x_{j_q} ^{n_{j_q}}. 
\end{align*}
When $r=0$ (resp. $r=m$), we do not need the first (resp. second) expansion. 
The convergences of these power series expansions are verified as follows. We explain only the first one.
By the construction of $\tilde{\De }_{i_{1} \cdots i_{r}} $, we have 
$$
x_{k} <\frac{\vep _{k}}{m} \vep ,\ |s_{i_p}| \geq \vep_{i_p} ,\ 
\Big| 1-\sum_q s_{j_q} \Big| \geq \vep .
$$
Thus the uniform convergence on $\tilde{\De }_{i_{1} \cdots i_{r}} $ follows from 
\begin{align*}
  \Bigg| \frac{1}{1-\sum_q s_{j_q}} \cdot \sum_p \frac{x_{i_p}}{s_{i_p}}   \Bigg| 
  \leq \frac{1}{|1-\sum_q s_{j_q}|} \cdot  \sum_p \frac{|x_{i_p}|}{|s_{i_p}|}   
  < \frac{1}{\vep} \cdot \sum_p \frac{1}{\vep_{i_p}} \frac{\vep_{i_p}}{m} \vep =\frac{r}{m}\leq 1.
\end{align*}
Since $\tilde{\De}_{i_1 \cdots i_r}$ is constructed as a finite sum of 
loaded (compact) simplexes, 
we can exchange the sum and the integral in the expression of $F_{i_1 \cdots i_r}$. 
Then the coefficient of $x_1 ^{n_1} \cdots x_m ^{n_m}$ in the series expansion of $F _{i_{1} \cdots i_{r}}$ is
\begin{align}
  B_{n_1 \dots n_m}:=& 
  \frac{\Ga (a-\sum c_{k} +m+\sum n_{i_p} )}{\Ga (a-\sum c_{k} +m) } 
	\cdot \frac{\Ga (b+\sum n_{j_q} )}{\Ga (b) } \cdot \prod_k \frac{1}{n_k !} \label{Bn} \\
  & \cdot \int _{\tilde{\De }_{i_{1} \cdots i_{r}}} 
        \prod_p s_{i_p} ^{c_{i_p} -n_{i_p} -2} \cdot (1-\sum_p s_{i_p})^{-b-\sum n_{j_q}} \nonumber \\ 
  & \hspace{10mm} \cdot \prod_q s_{j_q} ^{-c_{j_q} -n_{j_q}} 
        \cdot (1-\sum_q s_{j_q})^{\sum c_{k} -a-m-\sum n_{i_p}} 
        ds_1 \! \wedge \! \cdots \! \wedge \! ds_m . 
        \nonumber
\end{align}
By the construction, the twisted cycle $\tilde{\De }_{i_{1} \cdots i_{r}}$ of this integral is 
identified with the usual regularization of the domain 
\begin{align*}
  \left\{ (s_1 ,\ldots ,s_m )\in \R^m \ \Bigg| \ 
    \begin{array}{l}
      s_{i_p} >0 ,\ 1-\sum s_{i_p} >0 ,\\
      s_{j_q} >0 ,\ 1-\sum s_{j_q} >0
    \end{array}
  \right\} 
\end{align*}
for the multi-valued function 
$$
\prod_p s_{i_p} ^{c_{i_p} \! -\! n_{i_p} -1} (1-\sum s_{i_p})^{-\! b\! -\! \sum n_{j_q}+1} 
\cdot \prod_q s_{j_q} ^{-\! c_{j_q} \! -\! n_{j_q}+1} \cdot 
(1-\sum_q s_{j_q})^{\sum c_{k} \! -\! a\! -\! m\! -\! \sum n_{i_p}+1}
$$
on $\C^m -\left( \bigcup_k (s_k=0) \cup (1-\sum s_{i_p} =0) \cup (1- \sum s_{j_q} =0) \right)$. 
Hence the integral in (\ref{Bn}) is equal to
\begin{align*}
& \frac{\prod \Ga (c_{i_p}-n_{i_p}-1) \cdot \Ga (-b-\sum n_{j_q}+1)}{\Ga (-b+\sum c_{i_p} -\sum n_{k}-r+1)} \\
& \cdot \frac{\prod \Ga (-c_{j_q}-n_{j_q}+1) 
\cdot \Ga (\sum c_{k} -a-m-\sum n_{i_p}+1)}{\Ga (\sum c_{i_p} -a-\sum n_{k}-r+1)} .
\end{align*}
Using the formula 
\begin{align}
  \Ga (z) \Ga (1-z) =\frac{\pi}{\sin (\pi z)} , \label{Gamma}
\end{align}
we thus have 
\begin{align*}
  \frac{ B_{n_1 \dots n_m}}{A_{n_1 \dots n_m}} =&
  \prod _{p}  \Ga (c_{i_{p}} -1) \cdot \prod _{q}  \Ga (1-c_{j_{q}} ) \\
  & \cdot \frac{\Ga (\sum c_{k} -a-m+1) \Ga (1-b)}{\Ga (\sum c_{i_{p}} -a-r+1) \Ga (\sum c_{i_{p}} -b-r+1)},
\end{align*}
which implies the proposition. 
\end{proof}

We define a bijection $\iota _{i_{1} \cdots i_{r}} :M_{i_1 \cdots i_r} \rightarrow M$ by
$$
\iota _{i_{1} \cdots i_{r}} (s_{1} ,\ldots ,s_{m} ):=(t_{1} ,\ldots ,t_{m} );\ 
t_{i_{p}} =\frac{x_{i_{p}}}{s_{i_{p}}} ,\ t_{j_{q}} =s_{j_{q}} .
$$
For example, $\iota (=\iota _{\emptyset })$ is the identity map on $M=M_{\emptyset}$, and 
$\iota _{1 \cdots m}$ defines an involution on $M=M_{1\cdots m}$.
We state our first main theorem.

\begin{Th}\label{solutions-cycle}
We define a twisted cycle $\De _{i_1 \cdots i_r}$ in $M$ by 
\begin{align}
  \label{cycle-def}
  \De _{i_{1} \cdots i_{r}} :=(-1)^r (\iota _{i_1 \cdots i_r} )_{*} (\tilde{\De }_{i_1 \cdots i_r}) .
\end{align}
Then we have
\begin{align*}
& \int _{\De _{i_{1} \cdots i_{r}} } \prod t_{k} ^{-c_{k} } \cdot (1-\sum t_{k} )^{\sum c_{k} -a-m} 
	\cdot \left( 1-\sum \frac{x_{k}}{t_{k}} \right) ^{-b} dt_{1} \wedge \cdots \wedge dt_{m} \\
& \left( =\int_{\De_{i_1 \cdots i_r}} u \vph  \right)
=\prod _{p}  x_{i_{p}} ^{1-c_{i_{p}}} \cdot F _{i_{1} \cdots i_{r}}  ,
\end{align*}
and hence this integral corresponds to the local solution $f_{i_{1} \cdots i_{r}} $ to 
$E_{C} (a,b,c)$ given in Proposition \ref{solution}.
\end{Th}
\begin{proof}
Pull-back the integral of the left hand side by $\iota _{i_{1} \cdots i_{r}} $. 
Indeed, we have
$$
t_{i_p} =\frac{x_{i_p}}{s_{i_p}} ,\ dt_{i_p}=-\frac{x_{i_p}}{s_{i_p} ^2}ds_{i_p} 
$$
by the definition of $\iota _{i_{1} \cdots i_{r}} $. 
Note that the sign $(-1)^r$ arising from the pull-back of $dt_1 \wedge \cdots \wedge dt_m$ 
is canceled by that in (\ref{cycle-def}). 
The second claim is followed from the first equality and Proposition \ref{series-cycle}. 
\end{proof}

\begin{Rem}\label{cycle-orientation}
  \begin{enumerate}[(i)]
  \item The sign $(-1)^r$ in (\ref{cycle-def}) 
    implies that the orientation of the direct product $\iota _{i_1 \cdots i_r} (\si _{i_1 \cdots i_r})$ 
    of two simplexes in $\De _{i_{1} \cdots i_{r}} $ is coincide with that of $\R^m$ 
    induced from the natural embedding $\R^m \subset \C^m$. 
    % Note that this simplex is equal to $\iota _{i_1 \cdots i_r} (\si _{i_1 \cdots i_r})$ as sets. 
  \item The twisted cycle $\De$ (for $r=0$) equals to that mentioned in Proposition \ref{int-Euler}. 
  \end{enumerate}
\end{Rem}

The replacement $u \mapsto u^{-1} =1/u$ and the construction same as $\De _{i_1 \cdots i_r}$ give 
the twisted cycle $\De_{i_1 \cdots i_r} ^{\vee}$ 
which represents an element in $H_m (\CC_{\bu} (M,u^{-1}))$. 
We obtain the intersection numbers of the twisted cycles $\{ \De _{i_1 \cdots i_r} \}$ and 
$\{ \De _{i_1 \cdots i_r}^{\vee} \}$. 

\begin{Th}\label{H-intersection}
  \begin{enumerate}[(i)]
  \item For $I,J \! \subset \! \{ 1,\ldots ,m \}$ such that $I \! \neq \! J$, 
    we have $I_h (\De_I ,\De_J ^{\vee}) \! =\! 0$. 
  \item The self-intersection number of $\De _{i_1 \cdots i_r}$ is 
  \begin{align*}
    I_{h} (\De _{i_{1} \cdots i_{r}} ,\De _{i_{1} \cdots i_{r}}^{\vee} )
    = (-1) ^{r} \cdot \frac{\DS \prod _{q}  \ga _{j_{q} } \cdot \left( \al -\prod _{p}  \ga _{i_{p}} \right) 
      \left( \be -\prod _{p}  \ga _{i_{p}} \right) }
    {\DS \prod _{k}  (\ga _{k} -1) \cdot \left( \al -\prod _{k}  \ga _{k} \right) (\be -1)} .
  \end{align*}
  \end{enumerate}
\end{Th}
\begin{proof}
(i) Since $\De _{i_{1} \cdots i_{r}} $'s represent local solutions (\ref{series-sol}) 
to $E_{C} (a,b,c)$ by Theorem \ref{solutions-cycle}, 
this claim is followed from similar arguments to the proof of Lemma 4.1 in \cite{GM}. \\
(ii) By using $\iota _{i_{1} \cdots i_{r}}$, the self-intersection number of
$\De _{i_{1} \cdots i_{r}}$ is equal to that of $\tilde{\De }_{i_{1} \cdots i_{r}}$ 
with respect to the multi-valued function $u_{i_1 \cdots i_r}$. 
To calculate this, we apply results of M. Kita and M. Yoshida (see \cite{KY}). 
Since we construct the twisted cycle $\tilde{\De }_{i_{1} \cdots i_{r}}$ from 
the direct product $\sigma_{i_1 \cdots i_r}$ of two simplexes, 
the self-intersection number of $\tilde{\De }_{i_{1} \cdots i_{r}}$ 
is obtained as the product of those of simplexes. Thus we have
\begin{align*}
  I_{h} (\De _{i_{1} \cdots i_{r}} ,\De _{i_{1} \cdots i_{r}}^{\vee} )
  = \frac{1-\prod_{p}\ga_{i_p} \cdot \be^{-1}}{\prod_{p}(1\! -\! \ga_{i_p})\cdot (1\! -\! \be^{-1})}
  \cdot \frac{1-\prod_{q}\ga_{j_q}^{-1} \cdot \prod_{k}\ga_k \cdot \al^{-1}}
             {\prod_{q}(1\! -\! \ga_{j_q}^{-1})\cdot (1\! -\! \prod_{k}\ga_k \cdot \al^{-1})} .
\end{align*}
\end{proof}

\section{Intersection numbers of twisted cohomology groups}
In this section, we review twisted cohomology groups and 
the intersection form between twisted cohomology groups in our situation, 
and evaluate some self-intersection numbers of twisted cocycles. 

Recall that 
\begin{align*}
  & M=\C^m -\left( \bigcup_k (t_k =0) \cup (v=0) \cup (w=0) \right) , \\
  & u= \prod t_k ^{1-c_k +b} \cdot v^{\sum c_{k} -a-m+1} w ^{-b} .
\end{align*}
We consider the logarithmic $1$-form 
$$
\om :=d\log u =\frac{du}{u} .
$$
We denote the $\C$-vector space of smooth $k$-forms on $M$ by $\CE^k (M)$. 
We define the covariant differential operator
$\na_{\om} :\CE^k (M) \to \CE^{k+1} (M)$ by
$$
\na_{\om} (\psi):= d\psi +\om \wedge \psi ,\ \ \psi \in \CE^{k} (M).
$$
Because of $\na_{\om} \circ \na_{\om} =0$, we have a complex 
$$
\CE^{\bu} (M) :\cdots 
\overset{\na_{\om}}{\longrightarrow} \CE^k (M)
\overset{\na_{\om}}{\longrightarrow} \CE^{k+1}(M)
\overset{\na_{\om}}{\longrightarrow} \cdots ,
$$
and its $k$-th cohomology group $H^k(M,\na_{\om})$. 
It is called the $k$-th twisted de Rham cohomology group. 
An element of $\ker \na_{\om}$ is called a twisted cocycle.
By replacing $\CE^k (M)$ with the $\C$-vector space $\CE_c ^k (M)$ of 
smooth $k$-forms on $M$ with compact support, we obtain 
the twisted de Rham cohomology group 
$H_c ^k (M,\na_{\om})$ with compact support. 
By \cite{Cho}, we have $H^k(M,\na_{\om})=0$ for all $k \neq m$. 
Further, by Lemma 2.9 in \cite{AK}, 
there is a canonical isomorphism 
$$
\jmath : H^m (M,\na_{\om}) \to H_c ^m (M,\na_{\om}) .
$$ 
By considering $u^{-1} =1/u$ instead of $u$, we have 
the covariant differential operator $\na_{-\om}$ and 
the twisted de Rham cohomology group $H^k (M,\na_{-\om})$. 
The intersection form $I_c$ between $H^m (M,\na_{\om})$ and $H^m (M,\na_{-\om})$ 
is defined by 
$$
I_c (\psi ,\psi'):=\int_M \jmath (\psi) \wedge \psi' ,\quad 
\psi \in H^m(M,\na_{\om}) ,\ \psi' \in H^m (M,\na_{-\om}),
$$
which converges because of the compactness of the support of $\jmath (\psi)$.

By the Poincar\'{e} duality (see Lemma 2.8 in \cite{AK}), we have 
\begin{align}
  & \dim H_k (\CC_{\bu} (M,u))=0 \ \  (k \neq m), \label{vanishing}\\
  & \dim H_m (\CC_{\bu} (M,u))=\dim H^m(M,\na_{\om}). \nonumber
\end{align}
~
\begin{Prop}\label{dim}
  Let $x_1 ,\ldots ,x_m$ be generic. 
  \begin{enumerate}[(i)]
  \item We have $\dim H_m (\CC_{\bu} (M,u))=2^m$. 
  \item The twisted cycles $\{ \De_I \} _I$ form a basis of $H_m (\CC_{\bu} (M,u))$. 
  \item The integrations of $u\vph$ on twisted cycles give 
    an isomorphism between $H_m (\CC_{\bu} (M,u))$ and 
    the space of local solutions to $E_C (a,b,c)$. 
  \end{enumerate}
\end{Prop}
\begin{proof}
  We prove (i).  
  By (\ref{vanishing}) and Theorem 2.2 in \cite{AK}, we have 
  $$
  \dim H_m (\CC_{\bu} (M,u))=(-1)^m \chi (M),
  $$
  where $\chi (M)$ is the Euler characteristic of $M$. 
  It is sufficient to show that $\chi (M)=(-1)^m \cdot 2^m$. 
  Let $h\in \C[T_0 ,\ldots ,T_m]$ be a homogeneous polynomial defined by 
  $$
  h(T):=
  \prod_{i=0}^{m} T_i \cdot \left( \sum_{i=0}^{m} \prod_{j\neq i} T_j \right) ,
  $$
  and let $D(h)=\{ T \in \P^m \mid h(T)\neq 0 \}$. 
  Then we have $\chi (M) =\chi (D(h)-L)$ for 
  some generic hyperplane $L$ in $\P^m$. 
  We consider the gradient map
  $$
  {\rm grad} (h):D(h) \to \P^m ;\quad
  [T_0 :\cdots :T_m ] \mapsto 
  \left[ \frac{\pa h}{\pa T_0}(T) :\cdots : \frac{\pa h}{\pa T_m}(T) \right] .
  $$
  It is easy to see that the degree of ${\rm grad} (h)$ is equal to $2^m$. 
  By Theorem 1 in \cite{DP}, we obtain 
  $\chi (D(h)-L)=(-1)^m \deg ({\rm grad} (h))=(-1)^m \cdot 2^m$, which shows (i). 
  (The author thanks to J. Kaneko for pointing out this fact.)
  \\
  The claim (ii) follows from (i), since 
  the determinant of the intersection matrix $(I_h (\De_I ,\De_J ^{\vee}))$ 
  is not zero by Theorem \ref{H-intersection}.   
  \\
  We show (iii). 
  Let $Sol$ be the space of local solutions to $E_C (a,b,c)$. 
  By Theorem \ref{solutions-cycle}, 
  integrals of $u \vph$ on 
  linear combinations of $\De_{i_1 \cdots i_r}$'s are in $Sol$. 
  Then (ii) implies that the linear map 
  $$
  \Phi :H_m (\CC_{\bu} (M,u)) \to Sol \ ;\quad 
  C \mapsto \int_C u\vph 
  $$
  is defined. 
  Proposition \ref{solution} and Theorem \ref{solutions-cycle} imply 
  that $\Phi$ is surjective. Therefore $\Phi$ is isomorphic 
  because of $\dim H_m (\CC_{\bu} (M,u))=\dim Sol =2^m$.  
\end{proof}

We evaluate some intersection numbers of the twisted cocycles 
\begin{align*}
& \vph =\frac{dt_{1} \wedge \cdots \wedge dt_{m} }{\prod t_k \cdot (1-\sum t_{k} )} ,\\ 
&\vph' :=\frac{dt_1 \wedge \cdots \wedge dt_m}{vw}
=\frac{dt_1 \wedge \cdots \wedge dt_m}{\prod t_k \cdot (1-\sum t_k) \cdot (1-\sum \frac{x_k}{t_k})} .
\end{align*} 
\begin{Th}\label{C-intersection}
  \begin{enumerate}[(i)]
  \item The self-intersection number of $\vph$ is
    \begin{align*}
      &I_{c} (\vph ,\vph) \\
      &= (2\pi \sqrt{-1}) ^{m} \left( \frac{1}{\sum c_{k} \! -\! a\! -\! m\! +\! 1} 
        +\frac{1}{b\! +\! m\! -\! \sum c_{k} } \right) 
      \cdot \sum _{ \{ I^{(r)} \} } \prod _{r=1} ^{m-1} \frac{1}{b+r-\sum c_{i_{p} ^{(r)}} },
    \end{align*}
    where $\{ I^{(r)} \}$ runs sequences of subsets of $\{ 1,\ldots ,m \}$, which satisfy 
    $$
    \{ 1,\ldots ,m \} \supsetneq I^{(m-1)} \supsetneq \cdots \supsetneq I^{(2)} \supsetneq I^{(1)} \neq \emptyset ,
    $$
    and we write $I^{(r)} =\{ i_{1} ^{(r)} ,\ldots ,i_{r} ^{(r)} \}$. 
  \item We have $I_c (\vph ,\vph')=0$.
\end{enumerate}
\end{Th}
\begin{proof}
  (i) The hypersurfaces $H_1 ,\ldots ,H_m ,H, D$ do not form a normal crossing divisor, 
  because of $H_i \cap H_j \subset D$ for $i \neq j$. 
  Thus we blow up $\C^m$ along some intersections of hyperplanes so that 
  the pole divisor of $\vph$ is normally crossing. 
  Firstly we consider the blow up at the origin ($=H_{1} \cap \cdots \cap H_{m}$). 
  Secondly we blow up this along $H_{1} \cap \cdots \cap \check{H}_{k} \cap \cdots \cap H_{m}$ ($k=1,\ldots ,m$). 
  Repeat the blowing up process, lastly we blow up along $H_{i} \cap H_{j}$ ($1\leq i<j\leq m$). 
  For $i_1 <\cdots <i_r ,\ r \geq 2$, the exceptional divisor $E_{i_1 \cdots i_r}$ arising from 
  the blow up along $H_{i_1} \cap \cdots \cap H_{i_r}$ has the exponent 
  $$
  \sum_{p=1}^{r} (1-c_{i_p} +b) -(r-1)b =b+r-\sum_{p=1}^{r} c_{i_p}.
  $$
  By expressing $\vph$ in each coordinates system, results in \cite{M-form} 
  give the self-intersection number of $\vph$. \\
  (ii) 
  By the definition, the pole divisor of $\vph '$ does not contain the exceptional divisors. 
  Hence the pole divisors of $\vph$ and $\vph'$ do not have 
  $m$ or more common factors, which implies $I_c (\vph ,\vph')=0$. 
\end{proof}

\begin{Rem}
  Precisely speaking, to evaluate intersection numbers of twisted cocycles, 
  we should blow up a compactification of $\C^m$ ($\P^m$ or $(\P^1)^m$) so that 
  the pole divisor of $\om =d\log u$ is normally crossing. 
  Though the blowing up process in the above proof is not enough, it is shown that 
  the exceptional divisors arising from the other blowing up processes do not appear as 
  the components of the pole divisor of $\vph$. 
  Thus the proof is completed. 
\end{Rem}

\begin{Rem}
  It seems difficult to evaluate the self-intersection number $I_c (\vph' ,\vph')$. 
  For $m=2$ (i.e., Appell's $F_4$), this number 
  $(=\mathcal{I}_c (\vph_{x,4} ,\vph_{x,4})$ in \cite{GM}$)$ 
  is written by the parameters and the factor of the defining equation of the singular locus 
  (see \cite{GM}).   
  Furthermore, the author does not find $m$-forms which form a basis of $H^m (M,\na_{\om})$. 
\end{Rem}

\section{Twisted period relations} \label{section-TPR}
By Theorem \ref{H-intersection} and \ref{C-intersection}, 
we state our second main theorem. 
\begin{Th}[Twisted period relations] \label{TPR}
We have 
\begin{align}
\label{TPR1}
I_{c} (\vph ,\vph) 
= \sum _{I} \frac{1}{I_{h} (\De _{i_{1} \cdots i_{r}} ,\De _{i_{1} \cdots i_{r}}^{\vee} )}
\cdot g_{i_{1} \cdots i_{r}} \cdot g_{i_{1} \cdots i_{r}} ^{\vee} , \\
\label{TPR2}
I_{c} (\vph ,\vph') 
= \sum _{I} \frac{1}{I_{h} (\De _{i_{1} \cdots i_{r}} ,\De _{i_{1} \cdots i_{r}}^{\vee} )}
\cdot g_{i_{1} \cdots i_{r}} \cdot h_{i_{1} \cdots i_{r}} ^{\vee} ,
\end{align}
where 
\begin{align*}
g_{i_{1} \cdots i_{r}} =\int _{\De _{i_1 \dots i_r}} u\vph ,\ 
g_{i_{1} \cdots i_{r}} ^{\vee} = \int _{\De _{i_1 \dots i_r} ^{\vee}} u^{-1} \vph ,\ 
h_{i_1 \cdots i_r} ^{\vee} = \int _{\De _{i_1 \dots i_r} ^{\vee}} u^{-1} \vph' .
\end{align*}
Further, under the notations 
\begin{align*}
  & a_{i_1 \cdots i_r} :=a-\sum c_{i_p} +r ,\ 
  b_{i_1 \cdots i_r} :=b-\sum c_{i_p} +r ,\ 
  \check{c}^{i_1 \cdots i_r} :=(2,\ldots ,2)-c^{i_1 \cdots i_r} ,
\end{align*}
the equalities (\ref{TPR1}) and (\ref{TPR2}) are reduced to 
\begin{align}
  & \sum_{I} (-1)^r \frac{1\! -\! a_{i_1 \cdots i_r} }{b_{i_1 \cdots i_r}} 
  \cdot F_C (a_{i_1 \cdots i_r} ,b_{i_1 \cdots i_r} ,c^{i_1 \cdots i_r};x) 
  \cdot F_C (2\! -\! a_{i_1 \cdots i_r} ,-b_{i_1 \cdots i_r} ,\check{c}^{i_1 \cdots i_r};x)
  \label{TPR1-2} \\
  &=\frac{(1-a+b) \cdot \prod (1-c_k)}{b b_{1 \cdots m}}  
  \cdot \sum _{ \{ I^{(r)} \} } \prod _{r=1} ^{m-1} \frac{1}{b_{I^{(r)}}}, 
  \nonumber \\
  &\sum_{I} (\! -\! 1)^r (a_{i_1 \cdots i_r} \! -\! 1)
  \! \cdot \! F_C (a_{i_1 \cdots i_r} ,b_{i_1 \cdots i_r} ,c^{i_1 \cdots i_r};x) 
  \! \cdot \! F_C (2\! -\! a_{i_1 \cdots i_r} ,1\! -\! b_{i_1 \cdots i_r} ,\check{c}^{i_1 \cdots i_r};x) 
  \label{TPR2-2} \\
  & =0, \nonumber
\end{align}
respectively. 
\end{Th}
\begin{proof}
  Because of the compatibility of intersection forms and pairings obtained by integrations 
  (see \cite{CM}), 
  we obtain the equalities (\ref{TPR1}) and (\ref{TPR2}). 
  We show that (\ref{TPR1}) is reduced to (\ref{TPR1-2}). 
  By Proposition \ref{series-cycle} and Theorem \ref{solutions-cycle}, we have
  \begin{align*}
    g_{i_{1} \cdots i_{r}} =& \prod _{p=1} ^{r} \Ga (c_{i_{p}} -1) \cdot \prod _{q=1} ^{m-r} \Ga (1-c_{j_{q}} ) 
     \cdot \frac{\Ga (\sum c_{k} -a-m+1) \Ga (1-b)}
     {\Ga (\sum c_{i_{p}} \! -a-r+1) \Ga (\sum c_{i_{p}} \! -b-r+1)} \\
    & \cdot \prod _{p=1} ^{r} x_{i_{p}} ^{1-c_{i_{p}}} \cdot 
    F_C (a+r-\sum c_{i_p} ,b+r-\sum c_{i_p} ,c^{i_1 \cdots i_r};x).
  \end{align*}
  On the other hand, we can express $g_{i_{1} \cdots i_{r}} ^{\vee}$ like this  
  by the replacement 
  $$
  (a,b,c) \longmapsto (2-a,-b,(2,\ldots ,2)-c),
  $$ 
  since $u^{-1}\vph$ is written as 
  $$
  u^{-1}\vph = \prod t_{k} ^{c_{k}-2 } \cdot (1-\sum t_{k} )^{-\sum c_{k} +a+m-2} 
	\cdot \left( 1-\sum \frac{x_{k}}{t_{k}} \right) ^{b} dt_{1} \wedge \cdots \wedge dt_{m} .
  $$
  Thus we obtain 
  \begin{align*}
    g_{i_{1} \cdots i_{r}}^{\vee} =& 
    \prod _{p=1} ^{r} \Ga (1-c_{i_{p}}) \cdot \prod _{q=1} ^{m-r} \Ga (c_{j_{q}}-1 ) \\
    & \cdot \frac{\Ga (-\sum c_{k} +a+m-1) \Ga (1+b)}
     {\Ga (-\sum c_{i_{p}} +a+r-1) \Ga (-\sum c_{i_{p}} +b+r+1)} \\
    & \cdot \prod _{p=1} ^{r} x_{i_{p}} ^{c_{i_{p}}-1} \cdot 
    F_C (2\! -\! a\! -\! r\! +\! \sum c_{i_p} ,-b\! -\! r\! +\! \sum c_{i_p} ,(2,\dots ,2)\! -\! c^{i_1 \cdots i_r};x).
  \end{align*}
By the formula (\ref{Gamma}), we have 
\begin{align*}
  & \prod \Ga(c_k -1)\Ga(1-c_k) \cdot \frac{\Ga (\sum c_{k} -a-m+1)\Ga (-\sum c_{k} +a+m-1)}
      {\Ga (\sum c_{i_{p}} -a-r+1)\Ga (-\sum c_{i_{p}} +a+r-1)} \\
      &\cdot \frac{ \Ga (1-b)\Ga (1+b)}{\Ga (\sum c_{i_{p}} -b-r+1)\Ga (-\sum c_{i_{p}} +b+r+1)}\\
  % &=(2\pi \sqrt{-1})^m 
  %   \frac{b}{\prod (1-c_k ) \cdot (a+m-1-\sum c_k)}\cdot \frac{a+r-1-\sum c_{i_p}}{b+r-\sum c_{i_p}} \\
  %   &\ \ \cdot \frac{\prod \ga_{j_q}\cdot (\prod \ga_{i_p} -\al) (\prod \ga_{i_p} -\be)}
  %   {\prod (\ga_k -1) \cdot (\al-\prod \ga_k) (\be -1)} \\
  &=(2\pi \sqrt{-1})^m \cdot \frac{b}{\prod (1-c_k ) \cdot (a+m-1-\sum c_k)} \\ 
     &\ \ \cdot (-1)^r \cdot \frac{a+r-1-\sum c_{i_p}}{b+r-\sum c_{i_p}} 
     \cdot I_{h} (\De _{i_{1} \cdots i_{r}} ,\De _{i_{1} \cdots i_{r}}^{\vee} ) .
\end{align*}
Hence, under the notations 
$a_{i_1 \cdots i_r} ,\ b_{i_1 \cdots i_r}$ and $\check{c}^{i_1 \cdots i_r}$ ,
the equality (\ref{TPR1}) is reduced to 
\begin{align*}
& \left( \frac{1}{1-a_{1 \cdots m}} +\frac{1}{b_{1 \cdots m}} \right) 
	\cdot \sum _{ \{ I^{(r)} \} } \prod _{r=1} ^{m-1} \frac{1}{b_{I^{(r)}}} \\
&= \frac{b}{\prod (1-c_k ) \cdot (a_{1 \cdots m}-1)} \\
& \cdot \sum_{I} (-1)^r \frac{a_{i_1 \cdots i_r} \! -\! 1}{b_{i_1 \cdots i_r}} 
\! \cdot \! F_C (a_{i_1 \cdots i_r} ,b_{i_1 \cdots i_r} ,c^{i_1 \cdots i_r};x) 
\! \cdot \! F_C (2\! -\! a_{i_1 \cdots i_r} ,-b_{i_1 \cdots i_r} ,\check{c}^{i_1 \cdots i_r};x). 
\end{align*}
By multiplying $\DS \frac{(1-a_{1 \cdots m}) \cdot \prod (1-c_k )}{b}$, 
we obtain (\ref{TPR1-2}). 
The similar calculation shows that (\ref{TPR2}) is reduced to (\ref{TPR2-2}).  
\end{proof}
Note that (\ref{TPR1-2}) and (\ref{TPR2-2}) are generalizations of 
some equalities in Corollary 6.1 of \cite{GM}.

\section*{Acknowledgments}
The author thanks Professor Keiji Matsumoto for his useful advice and constant encouragement.
He also thanks Professor Jyoichi Kaneko for helpful comments. 
% He is grateful to the referee for suggesting some improvement in 
% the previous version of the article. 

\end{document}